\documentclass{amsart}
\usepackage[latin1]{inputenc}
\usepackage{subfigure}
\usepackage{amsmath,amssymb}
\usepackage{graphicx}
\usepackage{breqn}

\newtheorem{thm}{Theorem}
\newtheorem{defin}[thm]{Definition}
\newtheorem{prop}[thm]{Proposition}
\newtheorem{lem}[thm]{Lemma}
\newtheorem{claim}[thm]{Claim}
\newtheorem{coro}[thm]{Corollary}

\def\cA{\mathcal{A}}
\def\cB{\mathcal{B}}
\def\cF{\mathcal{F}}

\def\chM{\widehat{\mathcal{M}}}
\def\cH{\mathcal{H}}
\def\cK{\mathcal{K}}
\def\cR{\mathcal{R}}
\def\cT{\mathcal{T}}
\def\cB{\mathcal{B}}
\def\cQ{\mathcal{Q}}
\def\chB{\widehat{\mathcal{B}}}
\def\chQ{\widehat{\mathcal{Q}}}
\def\hB{\widehat{B}}
\def\hQ{\widehat{Q}}
\newcommand{\mytilde}{\raise.17ex\hbox{$\scriptstyle\mathtt{\sim}$}}

\author{Gwendal Collet\and \'Eric Fusy}
\thanks{LIX, \'Ecole Polytechnique, 91128 Palaiseau, France. \texttt{fusy,gcollet@lix.polytechnique.fr}.}

\begin{document}

\title[A formula for (quasi-)constellations with boundaries]{A simple formula for the series of constellations and quasi-constellations with boundaries}

\begin{abstract}
We obtain a very simple formula for the generating 
function of bipartite (resp. quasi-bipartite)
planar maps with boundaries (holes) of prescribed lengths, which generalizes
certain expressions obtained by Eynard in a book to appear. The formula is derived
from a bijection due to Bouttier, Di Francesco and Guitter combined
with a process (reminiscent of a construction of Pitman) 
of aggregating connected components of a forest into a single tree.
The formula naturally extends to $p$-constellations and quasi-$p$-constellations with boundaries (the case $p=2$ corresponding to bipartite maps).
\end{abstract}

 \maketitle

\section{Introduction}
\label{sec:in}
Planar maps, \textit{i.e.}, connected graphs embedded on the sphere, have attracted
a lot of attention since the seminal work of Tutte~\cite{Tutte62,Tutte63}. By considering rooted
maps (i.e., maps where a corner is marked~\footnote{In the literature, rooted maps are often
defined as maps with a marked oriented edge, which is equivalent to marking a corner, e.g., 
the corner to the left of the origin of the marked edge.}) 
and using a recursive approach, 
Tutte found beautiful counting formulas for many families of maps (bipartite, triangulations,...).
Several features occur recurrently (see~\cite{BJ05a} for a unified treatment): 
the generating function $y=y(x)$ is typically algebraic, often lagrangean
(i.e., there is a parametrization as $\{y=Q_1(t),x=Q_2(t)\}$, where $Q_1(.)$ and $Q_2(.)$ are explicit
rational expressions), yielding simple (binomial-like) formulas for the counting coefficients $c_n$,
and the asymptotics of the coefficients is in $c\ \!\gamma^nn^{-5/2}$ for some constants $c>0$ and $\gamma>1$. In this article we firstly focus on bipartite maps (all faces have even degree)
and on quasi-bipartite maps (all faces have even degree except for two, which have odd degree). 
 One of the first counting results obtained by Tutte is a strikingly simple formula (called
 formula of slicings) 
for the number $A[\ell_1,\ldots,\ell_r]$ of maps 
with $r$ numbered faces $f_1,\ldots,f_r$ of respective degrees $\ell_1,\ldots,\ell_r$,
  each face having a marked corner (for simple parity reasons the number of odd $\ell_i$
must be even).       

\begin{figure}
\begin{center}
\includegraphics[width=0.8\linewidth]{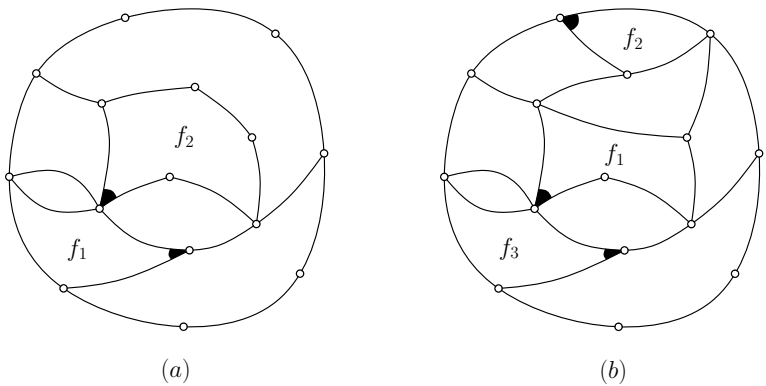}
\end{center}
\caption{(a)~A bipartite map with $2$ boundaries $f_1, f_2$ of respective degree $4, 6$. (b)~A quasi-bipartite map with $3$ boundaries $f_1, f_2, f_3$ of respective degree $5, 3, 4$.}
\label{fig:b_qb_maps}
\end{figure}

Solving a technically involved recurrence
satisfied by these coefficients, he proved in~\cite{Tutte62} 
that when none or only two of the $\ell_i$
are odd (bipartite and quasi-bipartite case, respectively), then: 
\begin{equation}\label{eq:slicings}
A[\ell_1,\ldots,\ell_r]=\frac{(e-1)!}{v!}\prod_{i=1}^r\alpha(\ell_i),\ \ \mathrm{with}\ \alpha(\ell):=\frac{\ell!}{\lfloor\ell/2\rfloor!\lfloor(\ell-1)/2\rfloor!},
\end{equation}
where $e=\sum_{i=1}^r\ell_i/2$ and $v=e-r+2$ are the numbers of edges and
vertices in such maps. The formula was recovered by Cori~\cite{Co75,Co76} 
(using a certain encoding procedure for planar maps); 
and the formula in the bipartite case was 
rediscovered bijectively by Schaeffer~\cite{Sc97}, based
on a correspondence with so-called \emph{blossoming trees}. 
Alternatively one can use a more recent bijection by Bouttier, Di Francesco and
Guitter~\cite{BDG04}  
(based on a correspondence with so-called \emph{mobiles}) which 
itself extends earlier constructions by Cori and Vauquelin~\cite{CoVa81} and by Schaeffer~\cite[Sec.~6.1]{S-these} for quadrangulations. 
The bijection with mobiles yields the following: if we denote
by  $R\equiv R(t)\equiv R(t;x_1,x_2,\ldots)$ the generating function specified by 
\begin{equation}\label{eq:seriesR}
R=t+\sum_{i\geq 1}x_i\binom{2i-1}{i}R^{i}.
\end{equation}
and denote by $M(t)\equiv M(t;x_1,x_2,\ldots)$ the generating function of rooted bipartite maps, where $t$ marks the number of vertices and $x_i$ marks the number of faces of degree $2i$
for $i\geq 1$, then  $M'(t)=2R(t)$. And one easily recovers~\eqref{eq:slicings} 
in the bipartite case 
by an application of the Lagrange
inversion formula to extract the coefficients of $R(t)$.

As we can see, maps might satisfy beautiful counting formulas, regarding \emph{counting coefficients}~\footnote{We also mention the work of Krikun~\cite{Kri07} where a beautiful
formula is proved for the number of 
triangulations with multiple boundaries of prescribed lengths, a bijective proof
of which is still to be found.}. 
Regarding \emph{generating functions}, formulas can be very nice and compact as well. 
In a book to be published~\cite{Eynard11}, Eynard gives an iterative procedure (based on residue calculations) to compute 
the generating function of maps of arbitrary genus and  with several marked faces, which we will call
boundary-faces (or shortly boundaries). In certain cases, this yields an explicit expression for the generating function. 
For example, he obtains formulas for the (multivariate) generating functions of bipartite and quasi-bipartite maps with two or three boundaries of arbitrary lengths $\ell_1, \ell_2, \ell_3$ (in the quasi-bipartite case two of
these lengths are odd), where $t$ marks the number of vertices and $x_i$ marks
the number of non-boundary faces of degree $2i$:

\vspace{-.5cm}

\begin{eqnarray}
\label{eq-eynard1}
G_{\ell_1,\ell_2} &=& \gamma^{\ell_1 + \ell_2} \sum_{j=0}^{\left\lfloor \ell_2/2\right\rfloor} (\ell_2-2j) \frac{\ell_1!\ell_2!}{j!(\frac{\ell_1-\ell_2}{2}+j)!(\frac{\ell_1+\ell_2}{2}-j)!(\ell_2-j)!}, 
\\
\label{eq-eynard2}
G_{\ell_1,\ell_2,\ell_3} &=& \frac{\gamma^{\ell_1+\ell_2+\ell_3-1}}{y'(1)}\left( \prod_{i=1}^{3} \frac{\ell_i!}{\left\lfloor \ell_i/2\right\rfloor ! \left\lfloor (\ell_i-1)/2 \right\rfloor !} \right).
\end{eqnarray}
In these formulas the series $\gamma$ and $y'(1)$ are closely related to $R(t)$, precisely 
 $\gamma^2=R(t)$ and one can check that $y'(1)=\gamma/R'(t)$.  

In the first part of this article, we obtain new formulas which generalize Eynard's ones to any number of boundaries, both in the bipartite and the quasi-bipartite case. For $r\geq 1$ and $\ell_1,\ldots,\ell_r$ positive integers, an \emph{even map} of type $(\ell_1,\ldots,\ell_r)$
is a map with $r$ (numbered) marked faces ---called \emph{boundary-faces}---   
$f_1,\ldots,f_r$ of degrees $\ell_1,\ldots,\ell_r$, each boundary-face having a marked
corner,  and with all the other faces of even degree. (Note that
there is an even number of odd $\ell_i$ by a simple parity argument.)  
Let $G_{\ell_1,\ldots,\ell_r}:=G_{\ell_1,\ldots,\ell_r}(t;x_1,x_2,\ldots)$ be the corresponding generating function 
where $t$ marks the number of vertices and $x_i$ marks
the number of non-boundary faces of degree $2i$. Our main result is:

\begin{thm}\label{theo:main}
When none or only two of the $\ell_i$ are odd,  
then the following formula holds:
\begin{equation}\label{eq:F}
G_{\ell_1,\ldots,\ell_r}=\Big(\prod_{i=1}^r\alpha(\ell_i)\Big)\cdot\frac1{s}\cdot\frac{\mathrm{d}^{r-\!2}}{\mathrm{d}t^{r-\!2}}R^s,
\end{equation}
$$\textrm{with}\ \alpha(\ell)=\frac{\ell!}{\lfloor \tfrac{\ell}{2}\rfloor!\lfloor \tfrac{\ell-1}{2}\rfloor!},\ s=\frac{\ell_1+\cdots+\ell_r}{2}, \textrm{ where } R \textrm{ is given by~\eqref{eq:seriesR}}.$$   
\end{thm}
Our formula covers all parity cases for the $\ell_i$ when $r\leq 3$. 
For $r=1$, the formula reads $G_{2a}\ \!\!'=\binom{2a}{a}R^a$, which is a direct consequence of the bijection with mobiles. 
For $r=2$ the formula reads 
$G_{\ell_1,\ell_2}=\alpha(\ell_1)\alpha(\ell_2)R^s/s$ (which simplifies the constant in~\eqref{eq-eynard1}). 
And for $r=3$ the formula reads $G_{\ell_1,\ell_2,\ell_3}=\alpha(\ell_1)\alpha(\ell_2)\alpha(\ell_3)R'R^{s-1}$. Note that~\eqref{eq:F} also ``contains" the 
formula of slicings~\eqref{eq:slicings}, by noticing that $A[\ell_1,\ldots,\ell_r]$ 
equals the evaluation of $G_{\ell_1,\ldots,\ell_r}$ at $\{t=1;x_1=0,x_2=0,\ldots\}$, which equals $(\prod_{i=1}^r\alpha(\ell_i))\cdot\frac{(s-1)!}{(s-r+2)!}$. 
Hence,~\eqref{eq:F} can be seen as an ``interpolation'' between the two formulas of Eynard  given above 
and Tutte's formula of slicings.   
In addition,~\eqref{eq:F} has the nice feature 
that the expression of $G_{\ell_1,\ldots,\ell_r}$ splits into two factors: (i) 
a \emph{constant factor} which itself is a product of independent contributions from every boundary,   
(ii) a \emph{series-factor} that just depends on the number of boundaries and the total length of the boundaries. 

Even though the coefficients of $G_{\ell_1,\ldots,\ell_r}$ have simple binomial-like expressions (easy to obtain from~\eqref{eq:slicings}), it does not explain 
why at the level of generating functions the expression~\eqref{eq:F} is so simple (and
it would not be obvious to guess ~\eqref{eq:F} by just looking at~\eqref{eq:slicings}). 
Relying
on the bijection with mobiles (recalled in Section~\ref{sec:first}), we give a transparent proof of~\eqref{eq:F}. In the bipartite case, our construction (described in Section~\ref{sec:bip}) 
starts from a forest of mobiles with some marked vertices, and then we aggregate 
the connected components so as to obtain a single mobile with some marked black vertices of fixed degrees (these black vertices correspond to the boundary-faces). 
The idea of aggregating connected components as we do is reminiscent of a construction due to Pitman~\cite{Pitman11},  
giving for instance 
a very simple proof (see~\cite[Chap.~26]{AigZie}) that the number of Cayley trees with $n$ nodes is $n^{n-2}$. 
Then we show in Section~\ref{sec:quasi} that the formula in the quasi-bipartite case can 
be obtained by a reduction to the bipartite case~\footnote{It would be interesting as a next step 
to search for a simple formula for $G_{\ell_1,\ldots,\ell_r}$ 
when four or more of the $\ell_i$ are odd (however, as noted by Tutte~\cite{Tutte62},
the coefficients do not seem to be that simple, they have large prime factors).} 
This reduction is done bijectively with the help of auxiliary trees called \emph{blossoming trees}. 
Let us mention that these blossoming trees have been introduced in another bijection with bipartite maps~\cite{Sc97}.  
We could alternatively use this bijection to prove Theorem~\ref{theo:main} in the bipartite
case (none of the $\ell_i$ is odd). But in order to encode quasi-bipartite maps, 
one would have to use extensions of 
this bijection~\cite{BMS02,BDG02} in which the encoding would become rather involved. 
This is the reason why we rely on bijections with mobiles, as given in~\cite{BDG04}.

In the second part of the article, we extend the formula of Theorem~\ref{theo:main} to constellations and quasi-constellations, families of maps which naturally generalize bipartite and quasi-bipartite maps. 
Define an \emph{hypermap} as an eulerian map (map with all faces of even degree)
 whose faces are bicolored ---there are dark faces and light faces--- such that any edge has a dark face on one side and a light face on the other side~\footnote{Hypermaps have several equivalent definitions in the literature; our definition
coincides with the one of Walsh~\cite{Nedela07}, by turning each dark face into a star centered at a dark vertex;
and coincides with the definitions of Cori and of James~\cite{SiWo08} 
where hypervertices are collapsed into vertices.}.  Define a \emph{$p$-hypermap} as a hypermap whose dark faces are of degree $p$ (note that classical maps
correspond to $2$-hypermaps, since each edge can be blown into a dark face of degree $2$). Note that the degrees of light faces in a $p$-hypermap add up to a multiple
of $p$.  A \emph{$p$-constellation}
is a $p$-hypermap such that the degrees of light faces are multiples of $p$, and a \emph{quasi $p$-constellation} is a $p$-hypermap such that exactly two light faces
have a degree not multiple of $p$. By the identification with maps, $2$-constellations and quasi $2$-constellations correspond respectively 
to bipartite maps and quasi-bipartite maps. 

\begin{figure}
\begin{center}
\includegraphics[width=0.8\linewidth]{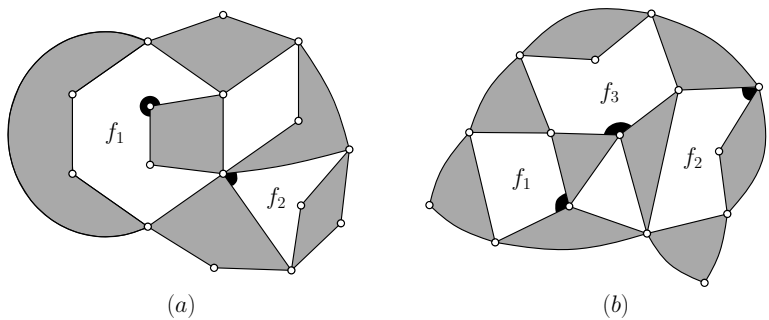}
\end{center}
\caption{(a)~A $4$-constellation with $2$ boundaries $f_1, f_2$ of respective degree $8, 4$. (b)~A quasi-$3$-constellation with $3$ boundaries $f_1, f_2, f_3$ of respective degree $4, 5, 6$.}
\label{fig:c_qc}
\end{figure}

Bouttier, Di Francesco and Guitter \cite{BDG04} also described a bijection for hypermaps, 
in correspondence with more involved mobiles (recalled in Section~\ref{sec:eulerian}). 
When applied to $p$-constellations, this bijection yields the following: if we denote by $R_p=R_p(t)=R_p(t;x_1,x_2,\ldots)$ the generating function specified by 
\begin{equation}\label{eq:seriesRp}
R_p=t+\sum_{i\geq 1}x_i\binom{pi-1}{i}R_p^{(p-1)i}.
\end{equation}
and by $C_p(t)=C_p(t;x_1,x_2,\ldots)$ the generating function of rooted $p$-constellations (\textit{i.e.}, $p$-constellations with a marked corner incident to a light face) where $t$ marks the number of vertices and $x_i$ marks the number of light faces of degree $pi$ for $i\geq 1$, then the bijection of~\cite{BDG04} ensures that $C_p'(t)=\frac{p}{p-1}R_p(t)$.

We use this bijection and tools from Sections~\ref{sec:bip} and~\ref{sec:quasi} to obtain the following formula for the generating function of  
constellations (proved in Section~\ref{sec:constel}) and quasi-constellations (proved in Section~\ref{sec:qconstel}). 
Let $G^{(p)}_{\ell_1,\ldots,\ell_r}:=G^{(p)}_{\ell_1,\ldots,\ell_r}(t;x_1,x_2,\ldots)$ 
be the generating function of $p$-hypermaps with $r$ (numbered) boundaries $f_1,\ldots,f_r$ of degrees $\ell_1,\ldots,\ell_r$, whose non-marked faces have degrees a 
multiple of $p$, 
where $t$ marks the number of vertices and $x_i$ marks
the number of non-boundary faces of degree $pi$. Then: 

\begin{thm}\label{theo:main2}
When none or only two of the $\ell_i$ are not multiple of $p$,  
then the following formula holds: 
\begin{equation}\label{eq:Fp}
G^{(p)}_{\ell_1,\ldots,\ell_r}=\Big(\prod_{i=1}^r\alpha(\ell_i)\Big)\cdot\frac{c}{s}\cdot\frac{\mathrm{d}^{r-\!2}}{\mathrm{d}t^{r-\!2}}R_p^s,\ 
\end{equation}
$\textrm{where}\ \displaystyle{\alpha(\ell)=\frac{\ell!}{\lfloor \ell/p \rfloor!\left(\ell - \lfloor\ell/p \rfloor -1\right)!}},\ s=\frac{p-1}{p}(\ell_1+\cdots+\ell_r),\ R_p \textrm{ is given by~\eqref{eq:seriesRp}},$\\
\textrm{and } $c=\left\{
\begin{array}{l}
  1, \textrm{~when every } \ell_i \textrm{~is a multiple of } p,\\
  p-1, \textrm{~when exactly two }\ell_i \textrm{~are not multiple of }p.
\end{array}
\right.$  
\end{thm}

First note that Theorem~\ref{theo:main} is the direct application of Theorem~\ref{theo:main2} when $p=2$.
Moreover, this yields the following extension of Tutte's slicing formula:

\begin{coro}
For $p\geq 2$, let $A^{(p)}[\ell_1,\ldots,\ell_r]$ be the number of $p$-hypermaps with exactly $r$ numbered light faces $f_1,\ldots,f_r$ of respective 
degrees $\ell_1,\ldots,\ell_r$, each light face having a marked corner.\\
When none or only two of the $\ell_i$ are not multiple of $p$ ($p$-constellations and quasi-$p$-constellations, respectively), then:

\begin{equation}\label{eq:slicings2}
A^{(p)}[\ell_1,\ldots,\ell_r]=c\frac{(\epsilon-d-1)!}{v!}\prod_{i=1}^r\alpha(\ell_i),\ \ \mathrm{with}\ \alpha(\ell):=\frac{\ell!}{\lfloor\ell/p\rfloor!(\ell-\lfloor\ell/p\rfloor-1)!},
\end{equation}
where $\epsilon=\sum_{i=1}^r\ell_i$ is the number of edges, $\displaystyle{d=\frac{\sum_{i=1}^r\ell_i}{p}}$ 
is the number of dark faces, and $v=\epsilon-d-r+2$ is the number of vertices,\\ 
and $c=\left\{
\begin{array}{l}
  1, \textrm{~when every } \ell_i \textrm{~is a multiple of } p,\\
 p-1, \textrm{~when exactly two }\ell_i \textrm{~are not multiple of }p.
\end{array}
\right.$  
\end{coro}

One gets~\eqref{eq:slicings2} out of~\eqref{eq:Fp} by taking the evaluation of $G^{(p)}_{\ell_1,\dots,\ell_r}$ at $\{t=1;x_1=0,x_2=0,\ldots\}$.
The expression of the numbers $A^{(p)}[\ell_1,\ldots,\ell_r]$ when all $\ell_i$ are multiples of $p$ 
has been discovered by Bousquet-M\'elou and Schaeffer \cite{BMS00}, but to our knowledge, the expression for quasi-constellations has not been given before 
(though it could also be obtained from Chapuy's results~\cite{Ch09}, see the paragraphs after Lemma~\ref{lem:middle} and Lemma~\ref{lem:pmiddle}).    

\vspace{.4cm}

\noindent {\bf Note.} This is the full version of a conference paper~\cite{CoFu12} 
entitled ``A simple formula for the series of bipartite and quasi-bipartite maps with boundaries'' 
presented at the conference FPSAC'12. In particular we extend here the formulas obtained in~\cite{CoFu12} to constellations and quasi-constellations.
We would like to mention that 
very recently Bouttier and Guitter~\cite{BG13} have found extensions of the formulas from~\cite{CoFu12} in another direction, to so-called $2b$-irreducible bipartite maps
(maps with all faces of degrees at least $2b$ and where all non-facial cycles have length at least $2b+2$).

\vspace{.4cm}

\noindent {\bf Notation.} We will often use the following notation: for $\cA$ and $\cB$ two 
(typically infinite) combinatorial classes and $a$ and $b$ two integers, 
write $a\cdot\cA\simeq b\cdot\cB$ if there is a ``natural'' 
$a$-to-$b$ correspondence between $\cA$ and $\cB$
(the correspondence 
will be explicit each time the notation is used) that preserves several parameters
(which will be listed when the notation is used, typically the correspondence will
preserve the face-degree distribution). 

\section{Bijection between vertex-pointed maps and mobiles}

\label{sec:first}

We recall here a well-known bijection due to Bouttier, Di Francesco and Guitter~\cite{BDG04}
between vertex-pointed planar maps and a certain family of decorated trees called mobiles.
We actually follow a slight reformulation of the bijection given in~\cite{BernardiFusy11}. 
A \emph{mobile} is a plane tree (i.e., a planar map with one face) with vertices either black or white, with 
dangling half-edges ---called buds--- at black vertices, such that there is no white-white edge, and such that each black vertex has as many buds as white neighbours.The \emph{degree} of a black vertex $v$ is the total number 
of incident half-edges (including the buds) incident to $v$. 
Starting from a planar map $G$ with a pointed vertex $v_0$, and where the vertices
of $G$ are considered as \emph{white}, one obtains a mobile $M$ as follows (see Figure~\ref{fig:b_mobiles}):

\begin{itemize}
\item
Endow $G$ with its geodesic orientation from $v_0$ (i.e., an edge $\{v,v'\}$ is oriented from $v$ to $v'$
if $v'$ is one unit further than $v$ from $v_0$, and is left unoriented if $v$ and $v'$ are at the same distance from $v_0$). 
\item
Put a new black vertex in each face of $G$. 
\item
Apply the following local rules to each edge (one rule for oriented edges
and one rule for unoriented edges) of $G$:  


\hspace{2.8cm}\includegraphics[width=4cm]{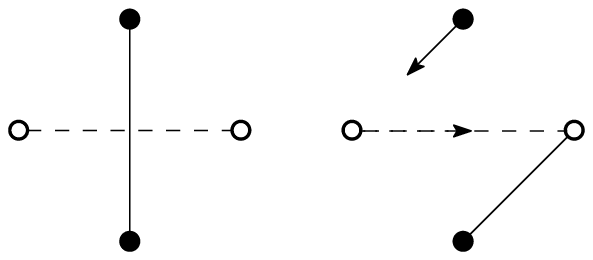}


\item
Delete the edges of $G$ and the vertex $v_0$. 
\end{itemize}

\begin{figure}
\begin{center}
\includegraphics[width=\linewidth]{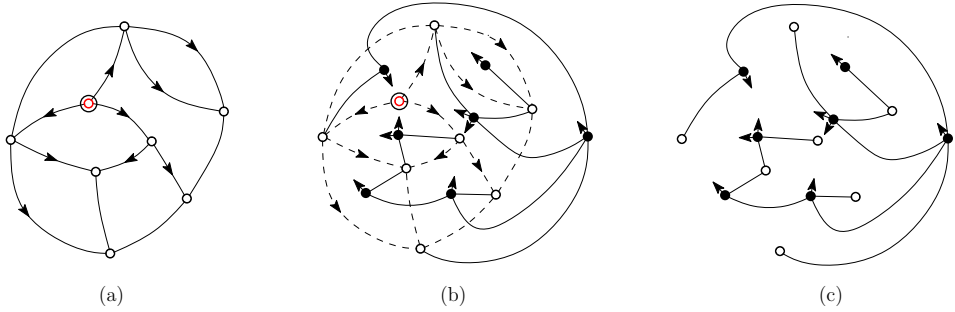}
\end{center}
\caption{(a)~A vertex-pointed map endowed with the geodesic orientation (with respect to the marked vertex).
(b) The local rule is applied to each edge of the map. (c)~The resulting mobile.}
\label{fig:b_mobiles}
\end{figure}

\begin{thm}[Bouttier, Di Francesco and Guitter~\cite{BDG04}]\label{theo:bdg}
The above construction is a bijection between vertex-pointed maps and mobiles. Each non-root vertex in the map corresponds to a white vertex in the mobile.
Each face of degree $i$ in the map corresponds to a black vertex of degree $i$ in the mobile. 
\end{thm}

A mobile is called \emph{bipartite} when all black vertices have even degree, and is called \emph{quasi-bipartite} when all black vertices have even degree except for two
which have odd degree. Note that bipartite (resp. quasi-bipartite) mobiles correspond to bipartite (resp. quasi-bipartite) vertex-pointed maps. 

\begin{claim}
A mobile is bipartite iff it has no black-black edge. A mobile is quasi-bipartite iff the set of black-black edges forms a non-empty path whose
extremities are the two black vertices of odd degrees. 
\end{claim}
\begin{proof}
Let $T$ be a mobile and $F$ the forest formed by the black vertices and black-black edges of $T$. 
Note that for each black vertex of  $T$, the degree and the number of incident black-black edges have same parity. 
Hence if $T$ is bipartite, $F$ has only vertices of even degree, so $F$ is empty; while if $T$ is quasi-bipartite, $F$ has
 two vertices of odd degree, so the only possibility is that the edges of $F$ form a non-empty path. 
\end{proof}

A bipartite mobile is called \emph{rooted} if it has a marked corner at a white vertex. 
Let $R:=R(t;x_1,x_2,\ldots)$ be the generating function of rooted bipartite mobiles, where
$t$ marks the number of white vertices and $x_i$ marks the number of black vertices
of degree $2i$ for $i\geq 1$.  As shown in~\cite{BDG04}, a decomposition at the root ensures that $R$ is given by 
Equation~\eqref{eq:seriesR}; indeed if we denote by $S$ the generating function of  bipartite mobiles rooted at
a white leaf, then $R=t+RS$ and $S=\sum_{i\geq 1}x_i\binom{2i-1}{i}R^{i-1}$. \\

For a mobile $\gamma$ with marked black vertices $b_1,\ldots,b_r$ of degrees $2a_1,\ldots,2a_r$, the associated \emph{pruned mobile} $\widehat\gamma$ 
  obtained from $\gamma$ by deleting the buds at the marked vertices (thus the marked vertices get degrees $a_1,\ldots,a_r$). Conversely,
such a pruned mobile yields  $\prod_{i=1}^r {2a_i-1 \choose a_i}$ mobiles (because of the number of ways to place the buds around the marked black vertices). 
Hence, if we denote by $\cB_{2a_1,\dots,2a_r}$ the family of bipartite mobiles with $r$ marked black vertices of respective degree $2a_1,\dots,2a_r$, 
and denote by $\chB_{2a_1,\dots,2a_r}$ the family of pruned bipartite mobiles with $r$ marked black vertices of respective degree $a_1,\dots,a_r$, 
we have:
$$\cB_{2a_1,\dots,2a_r} \simeq \prod_{i=1}^{r} {2a_i-1 \choose a_i}\chB_{2a_1,\dots,2a_r}.$$

\section{Bipartite case}\label{sec:bip}

In this section, we consider the two following families:
\begin{itemize}
	\item $\chM_{2a_1,\dots,2a_r}$ is the family of pruned bipartite mobiles with $r$ marked black vertices of respective degrees $a_1,\dots,a_r$, the mobile being rooted at a corner of one of the marked vertices,
	\item $\mathcal{F}_{s}$ is the family of forests made of $s:=\sum_{i=1}^{r} a_i$ rooted bipartite mobiles, and where additionnally $r-1$ white vertices $w_1,\ldots,w_{r-1}$ are marked. 
\end{itemize}

\begin{prop}
\label{bijection2}
There is an $(r-1)!$-to-$(r-1)!$ correspondence between the family $\chM_{2a_1,\dots,2a_r}$ and the family $\cF_{s}$.
If $\gamma\in\chM_{2a_1,\dots,2a_r}$ corresponds to $\gamma'\in\cF_s$, then each white vertex in $\gamma$ corresponds to a white vertex in $\gamma'$,
and each unmarked black vertex of degree $2i$ in $\gamma$ corresponds to a black vertex of degree $2i$ in $\gamma'$. 
\end{prop}

\begin{figure}
	\centering
	\includegraphics[width=\linewidth]{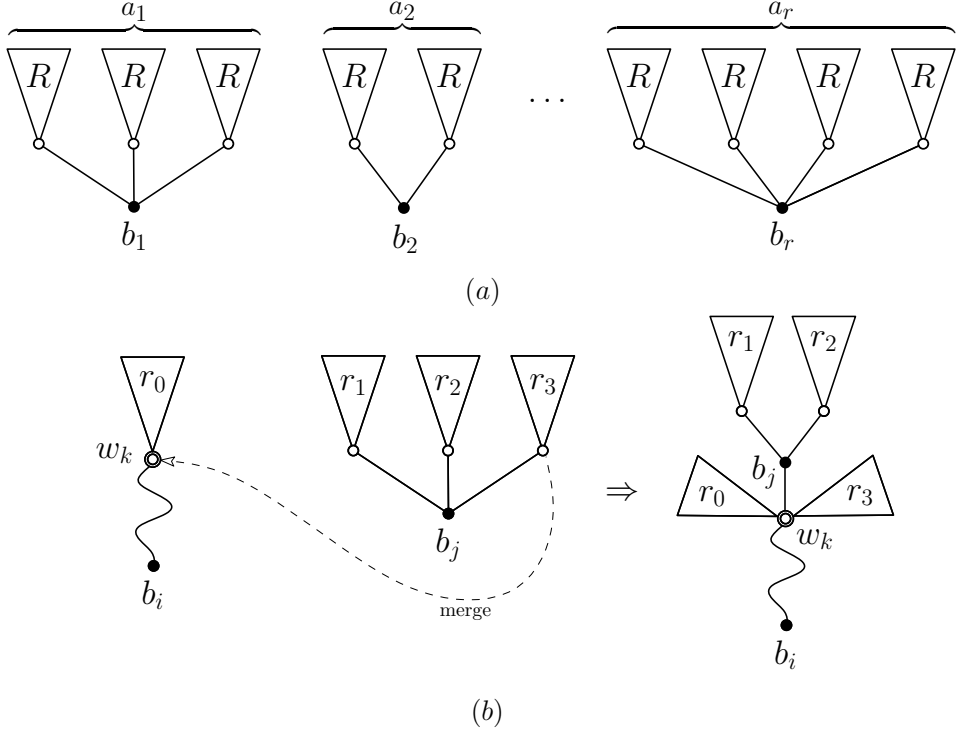}
	\caption{$(a)$ From a forest with $s=\sum_{i=1}^r{a_i}$ mobiles to $r$ components rooted at black vertices $b_1,\dots,b_r$. $(b)$ Merging the component rooted at $b_j$ with the distinct component rooted at $b_i$ containing the marked white vertex $w_k$.}
	\label{bij2}
\end{figure}

\begin{proof}
We will describe the correspondence in both ways (see Figure~\ref{bij2}).
First, one can go from the forest to the pruned mobile through the following operations:
\begin{enumerate}
	\item Group the first $a_1$ mobiles and bind them to a new black vertex $b_1$, then bind the next $a_2$ mobiles to 
a new black vertex $b_2$, and so on, to  get a forest with $r$ connected components rooted at $b_1,\dots,b_r$, see Figure~\ref{bij2}(a).
	\item The $r-1$ marked white vertices $w_1,\dots, w_{r-1}$ are ordered, pick one of the  $r-1$ components which do not contain $w_{r-1}$. Bind this component to $w_{r-1}$ by merging $w_{r-1}$ with the rightmost white neighbour of $b_i$, see Figure~\ref{bij2}(b). Repeat the operation for each $w_{r-i}$ to reduce the number of components to one ($r-i$ possibilities in the choice of the connected component at the $i$th step), thus getting a decorated bipartite tree rooted at a corner incident to some $b_j$,  and having $r$ black vertices  $b_1,\dots,b_r$ without buds.
\end{enumerate}

Conversely, one can go from the pruned mobile to the forest through the following operations:
\begin{enumerate}
	\item Pick one marked black vertex $b_k$, but the root, and separate it as in Figure \ref{bij2}(b) read from right to left. 
This creates a new connected component, rooted at $b_k$.  
	\item Repeat this operation, choosing at each step ($r-i$ possibilites 
	at the $i$th step)
	a marked black vertex that is not the root in its connected component, until
	one gets $r$ connected components, each being rooted at one of the marked
	black vertices $\{b_1,\ldots,b_r\}$ of respective degrees $a_1,\ldots,a_r$.  
	\item Remove all marked black vertices $b_1,\ldots,b_r$ and their incident edges; this yields
	a forest of $s$ rooted bipartite mobiles.
\end{enumerate}

In both ways, there are $\prod_{i=1}^{r-1}(r-i)=(r-1)!$ possibilities, that is,  the correspondence is $(r-1)!$-to-$(r-1)!$. 
\end{proof}

As a corollary we obtain the formula of Theorem~\ref{theo:main} in the bipartite case:
\begin{coro}\label{coro1}
For $r\geq 1$ and $a_1,\ldots,a_r$ positive integers, the generating function 
$G_{2a_1,\ldots,2a_r}$ satisfies~\eqref{eq:F}, i.e.,

\begin{equation}
\label{bip1}
G_{2a_1,\dots ,2a_r} = \left(\prod_{i=1}^r \frac{(2a_i)!}{a_i!(a_i-1)!}\right)
\cdot\frac1{s}\cdot\frac{\mathrm{d}^{r-\!2}}{\mathrm{d}t^{r-\!2}}R^s,\ \ \mathrm{where}\ s=\sum_{i=1}^r a_i.
\end{equation}
\end{coro}
\begin{proof}
As mentioned in the introduction, for $r=1$ the expression reads
 $G_{2a}\ \!\!'=\binom{2a}{a}R^a$, which is a direct consequence
of the bijection with mobiles (indeed $G_{2a}\ \!\!'$ is the series of mobiles
with a marked black vertex $v$ of degree $2a$, with a marked corner
incident to $v$). So we now assume $r\geq 2$. 
Let $B_{2a_1,\dots,2a_r}=B_{2a_1,\dots,2a_r}(t;x_1,x_2,\ldots)$ be the generating
function of $\cB_{2a_1,\dots,2a_r}$,
where $t$ marks the number of white vertices and $x_i$ marks the number of black
vertices of degree $2i$. Let $\widehat{M}_{2a_1,\dots,2a_r}=\widehat{M}_{2a_1,\dots,2a_r}(t;x_1,x_2,\ldots)$ 
be the generating function of $\chM_{2a_1,\dots,2a_r}$, where again $t$ marks the number of white 
vertices and $x_i$ marks the number of black vertices of degree $2i$. By definition
of $\chM_{2a_1,\dots,2a_r}$, we have:
$$
s\cdot B_{2a_1,\dots,2a_r}=\left(\prod_{i=1}^r \binom{2a_i-1}{a_i}\right)\cdot\widehat{M}_{2a_1,\dots,2a_r}
$$
where the factor $s$ is due to the number of ways to place the root (i.e., mark a corner at one of the marked black vertices), 
and the binomial product is due to the number of ways to place the buds around the marked black vertices. 
Moreover, Theorem~\ref{theo:bdg} ensures that: 
$$
G_{2a_1,\dots,2a_r}\ \!\!\!'=\left(\prod_{i=1}^r 2a_i\right)\cdot B_{2a_1,\dots,2a_r}
$$
where the multiplicative constant  is the consequence of a corner 
being marked in every boundary face, and where the derivative (according to $t$) is 
the consequence of a vertex being marked in the bipartite map. 
Next, Proposition \ref{bijection2} yields:
$$
\label{bip3}
\widehat{M}_{2a_1,\dots,2a_r}=\frac{\mathrm{d}^{r-\!1}}{\mathrm{d}t^{r-\!1}}R^s
$$
hence we conclude that:
$$
G_{2a_1,\dots,2a_r}\ \!\!\!'=\frac1{s}\left(\prod_{i=1}^r 2a_i \binom{2a_i-1}{a_i}\right)\cdot\frac{\mathrm{d}^{r-\!1}}{\mathrm{d}t^{r-\!1}}R^s,
$$
which, upon integration according to $t$, gives the claimed formula.
\end{proof}

\section{Quasi-bipartite case}\label{sec:quasi}
So far we have obtained an expression for 
the generating function $G_{\ell_1,\ldots,\ell_r}$ when all $\ell_i$ are even.
In general, by definition of even maps of type $(\ell_1,\ldots,\ell_r)$, 
there is an even number of $\ell_i$ of odd degree. We deal here with the case where exactly two of the $\ell_i$ are odd. 
This is done by a reduction to the bipartite case, 
using so-called \emph{blossoming trees} (already considered in~\cite{Sc97}) as auxililary structures, see Figure~\ref{fig:blossoming_tree}(a) 
 for an example.  
 
\begin{figure}
\begin{center}
\includegraphics[width=12cm]{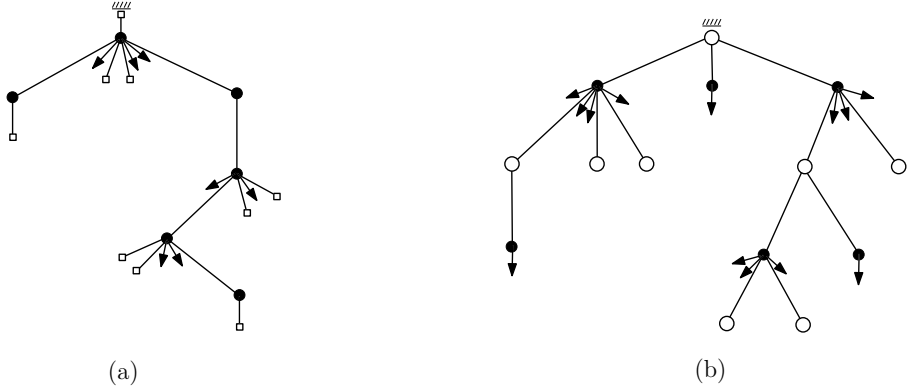}
\end{center}
\caption{(a)~A blossoming tree. 
(b) The corresponding rooted bipartite mobile.}
\label{fig:blossoming_tree}
\end{figure}

\begin{defin}[Blossoming trees]
A \emph{planted plane tree} is a plane tree with a marked leaf;
classically it is drawn in a top-down way; each vertex $v$ (different from the root-leaf) has $i$ (ordered) children, and the integer 
$i$ is called  the \emph{arity} of $v$. Vertices that are not leaves are colored
black (so a \emph{black vertex} means a vertex that is not a leaf).  
A \emph{blossoming tree} is a rooted plane tree where each black vertex $v$,
of arity $i\geq 1$,  
carries additionally $i-1$ dangling half-edges called \emph{buds} (leaves
 carry no bud). The \emph{degree} of such a black vertex $v$
is considered to be $2i$. 
\end{defin}

By a decomposition at the root, the generating function $T:=T(t;x_1,x_2,\dots)$ of blossoming trees, where $t$ marks the number of non-root leaves and $x_{i}$ marks 
the number of black vertices of degree $2i$, is given by:

\begin{equation}\label{eq:T}
T=t+\sum_{i \geq 1} x_{i} \binom{2i-1}{i} T^i.
\end{equation}

\begin{claim}\label{claim:TR}
There is a bijection between the family $\cT$ of blossoming trees and the family $\cR$ of rooted bipartite mobiles. 
For $\gamma\in\cT$ and $\gamma'\in\cR$ the associated rooted bipartite 
mobile, each non-root leaf of $\gamma$ corresponds to a white vertex of $\gamma'$,  and each black vertex of degree $2i$ in $\gamma$ corresponds to a black vertex of degree $2i$ in $\gamma'$. 
\end{claim}
\begin{proof}
Note that the decomposition-equation~\eqref{eq:T} satisfied by $T$ is exactly the same
as the decomposition-equation~\eqref{eq:seriesR} satisfied by $R$. Hence $T=R$, and 
one can easily produce recursively a bijection between $\cT$ and $\cR$ that
 sends black vertices of degree $2i$ to black vertices of degree $2i$, and sends leaves to white vertices, for instance 
Figure~\ref{fig:blossoming_tree} shows a blossoming tree and the corresponding rooted bipartite mobile.  
\end{proof}

The bijection between $\cT$ and $\cR$ will be used in order to get rid of the black path 
(between the two black vertices of odd degrees) 
which appears in a quasi-bipartite mobile. Note that, if we denote by $\cR'$ the family
of rooted mobiles with a marked white vertex (which does
not contribute to the number of white vertices), and by $\cT'$ the family 
of blossoming trees with a marked non-root leaf (which does
not contribute to the number of non-root leaves), then $\cT'\simeq\cR'$. 

Let $\tau$ be a  mobile with two marked black vertices $v_1,v_2$. 
Let $P=(e_1,\dots,e_k)$ be the path between $v_1$ and $v_2$ in $\tau$. If we untie $e_1$ from $v_1$ and $e_k$ from $v_2$, we obtain 3 connected components: the one containing $P$ is called the \emph{middle-part} $\tau'$ of $\tau$; 
the edges $e_1$ and $e_k$ are called respectively the \emph{first end} and 
the \emph{second end} of $\tau'$ in $\tau$. The vertices $v_1$ and $v_2$ are called \emph{extremal}.\\

Let $\cH$  be the family of structures that can be obtained as middle-parts of quasi-bipartite mobiles where $v_1$ and $v_2$ are the two black vertices of odd degree
(hence the path between $v_1$ and $v_2$ contains only black vertices).
And let $\cK$ be the family of structures  that can be obtained as middle-parts of  bipartite mobiles with two marked black vertices $v_1,v_2$.

\begin{lem}\label{lem:middle}
We have the following bijections:
$$ \cH \simeq \cT' \simeq \cR'  \hspace{3cm} \cK \simeq \cR' \times \cR $$
Hence: \begin{center} $\cK \simeq \cH \times \cR$. \end{center}

In these bijections each non-extremal black vertex of degree $2i$ 
in an object on the left-hand side corresponds to a non-extremal black vertex of degree $2i$ 
 in the corresponding object on the right-hand side.
\end{lem} 

\begin{figure}
 	\centering
 	\includegraphics[width=0.7\linewidth]{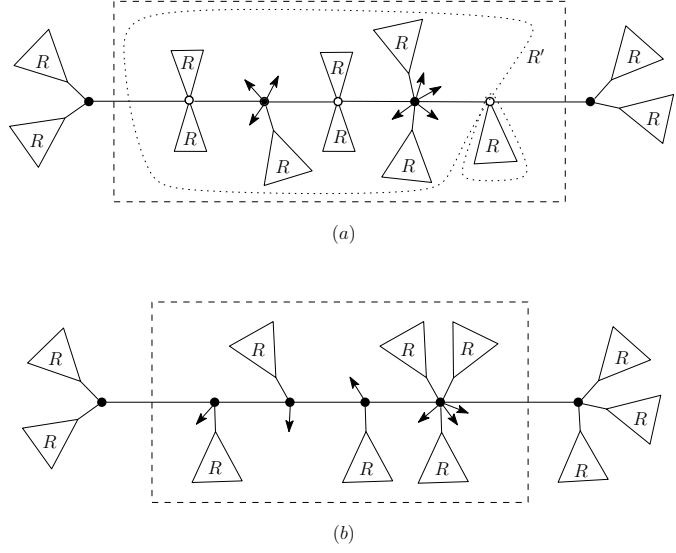}
 	\caption{middle-parts in the bipartite case $(a)$ and in the quasi-bipartite case $(b)$.}
	\label{fig:middle}
\end{figure}
	
\begin{proof}
Note that any $\tau \in \cH$  consists of a path $P$ 
of black vertices, and each vertex of degree $2i$ in $P$ carries (outside of $P$) $i-1$
buds and $i-1$ rooted mobiles (in $\cR$), as illustrated in Figure~\ref{fig:middle}(b). 
Let $\tau'$ be $\tau$ where each rooted mobile attached to $P$ 
is replaced by the corresponding blossoming tree (using the isomorphism 
of Claim~\ref{claim:TR}), and where the ends of $\gamma$ are considered
as two marked leaves (respectively the root-leaf and a marked
non-root leaf). We clearly have $\tau'\in\cT'$. Conversely, starting from $\tau'\in\cT'$,
let $P$ be the path between the root-leaf and the non-root marked leaf.
Each vertex of degree $2i$ on $P$ carries (outside of $P$) $i-1$ buds and $i-1$
blossoming trees. Replacing each blossoming tree attached to $P$
by the corresponding rooted mobile, and seeing the two marked leaves
as the first and second end of $P$, one gets a structure in $\cH$.
So we have $\cH\simeq\cT'$. 

The bijection $\cK\simeq\cR'\times\cR$ is simpler. Indeed, any $\tau\in\cK$
can be seen as a rooted mobile $\gamma$ 
 with a secondary marked corner at a white vertex (see Figure~\ref{fig:middle}(a)). 
Let $w$ (resp. $w'$) be the white vertex at the root (resp. at the secondary
marked corner) and let $P$ be the path between $w$ and $w'$. Each white vertex
on $P$ can be seen as carrying two rooted mobiles (in $\cR$), one on each
side of $P$. Let $r,r'$ be the two rooted mobiles attached at $w'$ (say, $r$
is the one on the left of $w'$ when looking toward $w$). 
If we untie $r$ from the rest of $\gamma$, then $w'$ now just acts as a marked
white vertex in $\gamma$, so the pair $(\gamma,r)$ is in $\cR'\times\cR$. 
The mapping from $(\gamma,r)\in\cR'\times\cR$ to $\tau\in\cK$ processes
in the reverse way. 
 We get $\cK\simeq\cR'\times\cR$.
\end{proof}
At the level of generating function expressions, 
Lemma~\ref{lem:middle} 
has been proved by Chapuy~\cite[Prop.7.5]{Ch09} in an even more precise form (which keeps track of a certain distance-parameter between the two extremities). 
We include our own proof to make the paper self-contained, and because the new idea of using blossoming trees as
auxiliary tools yields a short bijective proof. 

Now from Lemma~\ref{lem:middle} we can deduce a reduction from the quasi-bipartite
to the bipartite case (in Lemma~\ref{lem4} thereafter, see also Figure~\ref{fig:middle}). 
Let $a_1$ and $a_2$ be positive integers. 
Define $\cB_{2a_1,2a_2}$ as the family of bipartite mobiles with two marked black vertices $v_1, v_2$ of respective degrees $2a_1, 2a_2$. 
Similarly, define $\cQ_{2a_1-1,2a_2+1}$ as the family of quasi-bipartite mobiles with two marked black vertices $v_1, v_2$ of respective degrees $2a_1-1, 2a_2+1$ (i.e.,
the marked vertices are the two black vertices of odd degree).
Let $\chB_{2a_1,2a_2}$ be the family of pruned mobiles (recall that ``pruned'' means ``where buds at marked black vertices are taken out'') 
obtained from mobiles in $\cB_{2a_1,2a_2}$,
and let $\chQ_{2a_1-1,2a_2+1}$ be the family of pruned mobiles obtained from mobiles in $\cQ_{2a_1-1,2a_2+1}$.  

\begin{lem}\label{lem4}
For $a_1,a_2$ two positive integers:
$$\chB_{2a_1,2a_2} \simeq \chQ_{2a_1-1,2a_2+1}.$$
In addition, if $\gamma\in\chB_{2a_1,2a_2}$ corresponds to $\gamma'\in \chQ_{2a_1-1,2a_2+1}$, then each non-marked black vertex of degree $2i$ 
(resp. each white vertex)
in $\gamma$ corresponds to a non-marked black vertex of degree $2i$ 
(resp. to a white vertex) in $\gamma'$.
\end{lem}

\begin{proof}
Let $\gamma \in \chQ_{2a_1-1,2a_2+1}$, 
and let $\tau$ be the middle-part of $\gamma$. 
We construct $\gamma'\in\chB_{2a_1,2a_2}$ 
as follows. Note that $v_2$ has a black neighbour $b$ (along
the branch from $v_2$ to $v_1$) and has otherwise $a_2$ white neighbours. Let $w$
be next neighbour after $b$ in counter-clockwise order around $v_2$, and
let $r$ be the mobile (in $\cR$) hanging from $w$.  
According to Lemma~\ref{lem:middle}, 
the pair $(\tau,r)$ corresponds to some $\tau'\in\cK$. 
If we replace the middle-part $\tau$ by $\tau'$ and take out the edge $\{v_2,w\}$
and the mobile $r$, we obtain some $\gamma'\in\chB_{2a_1,2a_2}$.
The inverse process is easy to describe, so we obtain a bijection
between $\chQ_{2a_1-1,2a_2+1}$ and $\chB_{2a_1,2a_2}$. 
\end{proof}
Lemma~\ref{lem4} (in an equivalent form) 
has first been shown by Cori~\cite[Theo.VI~p.75]{Co75} (again we have provided our own short proof to be self-contained). 

As a corollary of Lemma~\ref{lem4}, we obtain the formula of Theorem~\ref{theo:main} in the quasi-bipartite case, with the exception of the case where 
the two odd boundaries are of length $1$ (this case will be treated later,
in Lemma~\ref{lem:Q11}). 

\begin{coro}\label{coro2}
For $r\geq 2$ and $a_1,\ldots,a_r$ positive integers, the generating
function $G_{2a_1-1,2a_2+1,2a_3,\dots,2a_r}$ satisfies~\eqref{eq:F}. 
\end{coro}
\begin{proof}
We first consider the case $r=2$. Let $\hB_{2a_1,2a_2}=\hB_{2a_1,2a_2}(t;x_1,x_2,\ldots)$ (resp. $B_{2a_1,2a_2}=B_{2a_1,2a_2}(t;x_1,x_2,\ldots)$)  
 be the generating function of $\chB_{2a_1,2a_2}$ (resp. of $\cB_{2a_1,2a_2}$)
 where $t$ marks the number of white vertices and $x_i$ marks the number of non-marked black vertices of degree $2i$.
There are $\binom{2a_i-1}{a_i}$ ways to place the buds at each marked black vertex
$v_i$ ($i\in\{1,2\}$), hence:
 $$
 B_{2a_1,2a_2}=\binom{2a_1-1}{a_1}\binom{2a_2-1}{a_2}\hB_{2a_1,2a_2}.
 $$  
 In addition Theorem~\ref{theo:bdg} ensures that $G_{2a_1,2a_2}\ \!\!'=2a_12a_2B_{2a_1,2a_2}$ (the multiplicative factor being due to the choice of a marked corner in each boundary-face). Hence:
 $$
 G_{2a_1,2a_2}\ \!\!'=4a_1a_2\binom{2a_1-1}{a_1}\binom{2a_2-1}{a_2}\hB_{2a_1,2a_2}.
 $$
 Similarly, if we denote by $\hQ_{2a_1-1,2a_2+1}=\hQ_{2a_1-1,2a_2+1}(t;x_1,x_2,\ldots)$ the generating function of the family $\chQ_{2a_1-1,2a_2+1}$  where $t$ marks the number of white vertices and $x_i$ marks the number of non-marked black vertices of degree $2i$,
 then we have:
 $$
 G_{2a_1-1,2a_2+1}\ \!\!'=(2a_1-1)(2a_2+1)\binom{2a_1-2}{a_1-1}\binom{2a_2}{a_2}\hQ_{2a_1-1,2a_2+1}.
 $$ 
 Since $\hB_{2a_1,2a_2}=\hQ_{2a_1-1,2a_2+1}$ by Lemma~\ref{lem4}, we get (with the notation $\alpha(\ell)=\displaystyle{\tfrac{\ell!}{\lfloor\ell/2\rfloor!\lfloor(\ell-1)/2\rfloor!}}$):
 $$
 \alpha(2a_1-1)\cdot\alpha(2a_2+1)\cdot G_{2a_1,2a_2}=\alpha(2a_1)\cdot \alpha(2a_2)\cdot G_{2a_1-1,2a_2+1}.
 $$
 In a very similar way (by the isomorphism of Lemma~\ref{lem4}), we have for $r\geq 2$:
 $$
  \alpha(2a_1-1)\cdot\alpha(2a_2+1)\cdot G_{2a_1,2a_2,2a_3,\ldots,2a_r}=\alpha(2a_1)\cdot \alpha(2a_2)\cdot G_{2a_1-1,2a_2+1,2a_3,\ldots,2a_r}.
 $$
 Hence the fact that $G_{2a_1-1,2a_2+1,2a_3,\ldots,2a_r}$ satisfies~\eqref{eq:F} 
 follows from the fact (already proved in Corollary~\ref{coro1}) that 
 $G_{2a_1,2a_2,2a_3,\ldots,2a_r}$ satisfies~\eqref{eq:F}. 
\end{proof}

It remains to show the fomula when the two odd boundary-faces have length $1$.
For that case, we have the following counterpart of Lemma~\ref{lem4}:

\begin{lem}\label{lem:Q11}
Let $\cB_{2}$ be the family of bipartite mobiles with a marked black vertex
of degree $2$, and let $\cB_2'$ be the family of objects from $\cB_2$ where
a white vertex is marked. Then
$$
\cQ_{1,1}\simeq\cB_2'.
$$
In addition, if $\gamma\in\cB_{2}'$ corresponds to $\gamma'\in \cQ_{1,1}$, then each white vertex of $\gamma$ corresponds to a white
vertex of $\gamma'$, and each non-marked black vertex of degree $2i$
in $\gamma$ corresponds to a non-marked black vertex of degree $2i$ in $\gamma'$.
\end{lem}
\begin{proof}
A mobile in $\cQ_{1,1}$ is completely reduced to its middle-part, so we have
$$
\cQ_{1,1}\simeq \cH\simeq \cT'\simeq\cR'.
$$
Consider a mobile in $\cR'$, i.e., a bipartite mobile where a corner incident to a white vertex is marked, and a secondary white vertex is marked. 
At the marked corner we can attach an edge connected to a new marked black vertex $b$ of degree $2$ (the other
incident half-edge of $b$ being a bud). We thus obtain a mobile in $\cB_2'$,
and the mapping is clearly a bijection.
\end{proof} 

 By Lemma~\ref{lem:Q11} we have $2\ \!G_{1,1}=G_{2}'$, and similarly
 $2\ \!G_{1,1,2a_3,\ldots,2a_r}=G_{2,2a_3,\ldots,2a_r}\ \!\!'$. Hence, again
 the fact that $G_{1,1,2a_3,\ldots,2a_r}$ satisfies~\eqref{eq:F} follows
 from the fact that $G_{2,2a_3,\ldots,2a_r}$ satisfies~\eqref{eq:F}, which
 has been shown in Corollary~\ref{coro1}.

\section{Extension to $p$-constellations and quasi $p$-constellations}
We show in this next section that the formula obtained for bipartite and quasi-bipartite maps (Theorem~\ref{theo:main}) 
naturally extends to a formula (Theorem~\ref{theo:main2}) for  
$p$-constellations and quasi $p$-constellations.  The ingredients are the same (bijection with mobiles and aggregation process to get the 
formula for $p$-constellations, and then 
use blossoming trees to reduce the formula for quasi $p$-constellations to the formula for $p$-constellations). 

\subsection{Bijection between vertex-pointed hypermaps and hypermobiles}\label{sec:eulerian}
Hypermaps admit a natural orientation by orienting each edge so as to have its incident dark face to its left. 
The following bijection is again a reformulation of the bijection in \cite{BDG04} between vertex-pointed eulerian maps and mobiles. 
Starting from a hypermap $G$ with a pointed vertex $v_0$, and where the vertices of $G$ are considered as \textit{round vertices}, 
one obtains a mobile $M$ as follows:

\begin{itemize}
	\item Endow $G$ with its natural orientation.
	\item Endow $G$ with its geodesic orientation by keeping oriented edges which belong to a geodesic oriented path from $v_0$.
	\item Label vertices of $G$ by their distance from $v_0$.
	\item Put a light (resp. dark) square in each light (resp. dark) face of $G$.
	\item Apply the following rules to each edge (oriented or not) of $G$:
\begin{center}
\includegraphics[width=7cm]{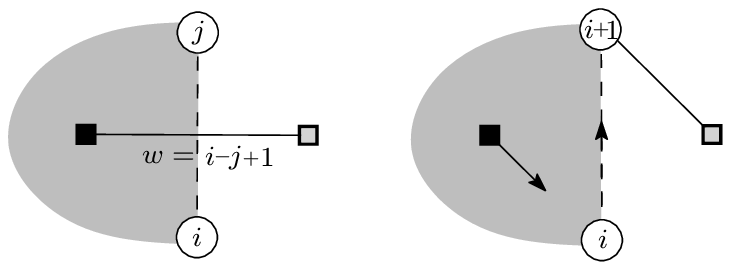}
\end{center}
	\item Forget labels on vertices.
\end{itemize}

\begin{figure}
\begin{center}
\includegraphics[width=\linewidth]{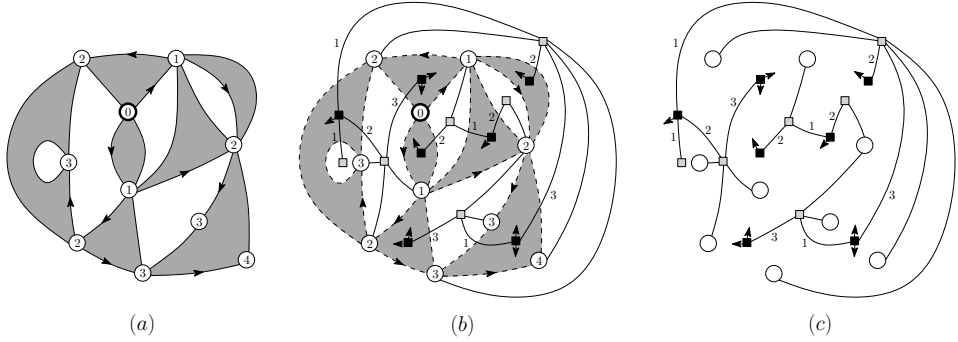}
\end{center}
\caption{(a)~A vertex-pointed hypermap endowed with its geodesic orientation (with respect to the marked vertex).
(b) The local rule is applied to each edge of the hypermap. (c)~The resulting hypermobile.}
\label{fig:b_mobiles3}
\end{figure}

\begin{defin}[Hypermobiles]
A hypermobile is a tree with three types of vertices (round, dark square, and light square) and positive integers (called \emph{weights}) on some edges, 
such that:

\begin{itemize} 
	\item there are two types of edges: between a round vertex and a light square vertex, or between a dark square vertex and a light square vertex (these edges
are called \emph{dark-light edges}),  
        \item dark square vertices possibly carry buds,
	\item dark-light edges carry a strictly positive weight, such that, for each square vertex (dark or light), 
the sum of weights on its incident edges equals the degree of the vertex.
\end{itemize}
\end{defin}

\begin{thm}[Bouttier, Di Francesco and Guitter~\cite{BDG04}]\label{theo:bdg2}
The above construction is a bijection between vertex-pointed hypermaps and hypermobiles. 
Each non-pointed vertex in the hypermap corresponds to a round vertex in the associated hypermobile,  
and each dark (resp. light) face corresponds to a dark (resp. light) square vertex of the same degree in the associated hypermobile.
\end{thm}

\subsection{Proof of Theorem~\ref{theo:main2} for $p$-constellations}\label{sec:constel}
For $p\geq 2$, hypermobiles corresponding to vertex-pointed $p$-constellations are called \textit{$p$-mobiles}.  
\begin{claim}[Characterization of $p$-mobiles~\cite{BDG04}]\label{claim:pmob}
A $p$-mobile satisfies the following properties:
\begin{itemize}
	\item dark-light edges have weight $p$,
	\item each dark square vertex, of degree $p$, has one light square neighbour and $p-1$ buds (thus can be seen as a ``big bud'' attached to the light square neighbour), 
	\item each light square vertex, of degree $pi$ for some $i\geq 1$, has $i$ dark square neighbours (i.e., carries $i$ big buds) and $(p-1)i$ round neighbours. 
\end{itemize}
\end{claim}

\begin{proof}
The first assertion is proved as follows.  
Let $T$ be a $p$-mobile and $F$ the forest formed by the edges whose weight is not a multiple of $p$, and their incident vertices. By construction, for each vertex of $T$, the degree and the sum of weights are multiple of $p$. Assume $F$ is non-empty. Then $F$ has a leaf $v$. Hence $v$ has a unique incident edge whose weight is not a multiple of $p$,
which implies that the degree of $p$ is not a multiple of $p$, a contradiction. Hence $F$ is empty and each weight in $T$ is a multiple of $p$. 
Moreover, dark square vertices have degree $p$, which implies that weights are at most equal to $p$. Hence all weights are equal to $p$.  
Then the second and third assertion follow directly from the first one. 
\end{proof}

Since the weights are always $p$ they can be omitted, and seeing dark square vertices as ``big buds'' it is clear that in the case $p=2$ we recover the 
mobiles for bipartite maps.  
A \emph{rooted $p$-mobile} is a $p$-mobile with a marked corner at a round vertex. Let $R_p\equiv R_p(t)\equiv R_p(t;x_1,x_2,\ldots)$ be the generating function of 
rooted $p$-mobiles where $t$ marks the number of white vertices and, for $i\geq 1$, $x_i$ marks the number of light square vertices of degree $pi$.
By a decomposition at the root (see~\cite{BDG04}), $R_p$  satisfies~\eqref{eq:seriesRp}.

\begin{figure}
\begin{center}
\includegraphics[width=\linewidth]{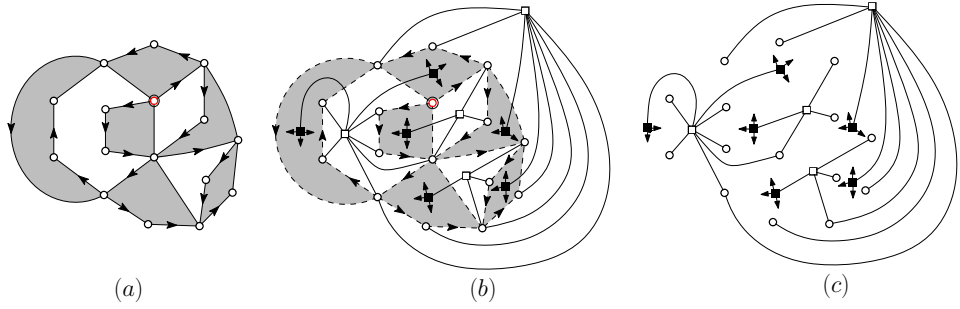}
\end{center}
\caption{(a)~A vertex-pointed $p$-constellation, $p=4$, endowed with its geodesic orientation (with respect to the marked vertex).
(b) The local rule is applied to each edge of the map. (c)~The resulting $p$-mobile (weights on dark-light edges, which all equal $p$, are omitted).}
\label{fig:b_pmob}
\end{figure}
  
One can now use the same processus as in Section \ref{sec:bip} to describe $p$-constellations with $r$ boundaries. 
For a $p$-mobile $\gamma$ with marked light square vertices $b_1,\ldots,b_r$ of degrees $pa_1,\ldots,pa_r$, the associated \emph{pruned $p$-mobile} $\widehat\gamma$ 
  is obtained from $\gamma$ by deleting the (big) buds at the marked vertices (thus the marked vertices get degrees $(p-1)a_1,\ldots,(p-1)a_r$). Conversely,
such a pruned mobile yields  $\prod_{i=1}^{r} {pa_i-1 \choose a_i}$ mobiles (because of the number of ways to place the big buds around the marked light square vertices). 
Hence, if we denote by $\cB^{(p)}_{pa_1,\dots,pa_r}$ the family of $p$-mobiles with $r$ marked light square vertices of respective degrees $pa_1,\dots,pa_r$, 
and denote by $\chB_{pa_1,\dots,pa_r}^{(p)}$ the family of pruned $p$-mobiles with $r$ marked light square vertices of respective degree $(p-1)a_1,\dots,(p-1)a_r$, 
we have:
$$\cB^{(p)}_{pa_1,\dots,pa_r} \simeq \prod_{i=1}^{r} {pa_i-1 \choose a_i}\chB^{(p)}_{pa_1,\dots,pa_r} $$
We consider the two following families:
\begin{itemize}
	\item $\chM^{(p)}_{pa_1,\dots,pa_r}$ is the family of pruned $p$-mobiles with $r$ marked light square vertices $v_1,\ldots,v_r$  of respective degrees $(p-1)a_1,\dots,(p-1)a_r$, the mobile being rooted at a corner of one of the marked vertices,
	\item $\mathcal{F}^{(p)}_{s}$ is the family of forests made of $s:=(p-1)\sum_{i=1}^{r} a_i$ rooted $p$-mobiles, and where additionnally $r-1$ round 
vertices $w_1,\ldots,w_{r-1}$ are marked. 
\end{itemize}

\begin{prop}
\label{bijection3}
There is an $(r-1)!$-to-$(r-1)!$ correspondence between the family $\chM^{(p)}_{pa_1,\dots,pa_r}$ and the family $\cF^{(p)}_{s}$.
If $\gamma\in\chM^{(p)}_{pa_1,\dots,pa_r}$ corresponds to $\gamma'\in\cF^{(p)}_s$, then each round vertex in $\gamma$ corresponds to a round vertex in $\gamma'$, and each light square vertex of degree $pi$ in $\gamma$ corresponds to a light square vertex of degree $pi$ in $\gamma'$. 
\end{prop}
\begin{proof}
This correspondence works in the same way as in Theorem~\ref{bijection2}, where light square vertices act as black vertices and round vertices act as white vertices, and
where one groups the first $(p-1)a_1$ components of the forest, then the following  $(p-1)a_2$ components, and so on, and then uses the same aggregation process
as in the bipartite case. 
\end{proof}

As a corollary we obtain the formula of Theorem~\ref{theo:main2} in the case of $p$-constellations:
\begin{coro}\label{coro3}
For $r\geq 1$ and $a_1,\ldots,a_r$ positive integers, the generating function 
$G^{(p)}_{pa_1,\ldots,pa_r}$ satisfies:
\begin{equation}
\label{pmob1}
G^{(p)}_{pa_1,\dots ,pa_r} = \left(\prod_{i=1}^r \frac{(pa_i)!}{((p-1)a_i-1)!a_i!}\right)
\cdot\frac1{s}\cdot\frac{\mathrm{d}^{r-\!2}}{\mathrm{d}t^{r-\!2}}R_p^s,\ \ \mathrm{where}\ s=(p-1)\sum_{i=1}^r a_i.\end{equation}
\end{coro}
\begin{proof}
In the case $r=1$, the expression reads $G^{(p)}_{pa}\ \!\!'=\binom{pa}{a}R_p^a$, which is a direct consequence
of the bijection with $p$-mobiles (indeed $G^{(p)}_{pa}\ \!\!'$ is the series of $p$-mobiles
with a marked light square vertex $v$ of degree $pa$, with a marked corner
incident to $v$). So we now assume $r\geq 2$. 
The formula derives (as formula~\eqref{bip1}) by combining the bijection of Theorem~\ref{theo:bdg2} and the correspondence of Proposition \ref{bijection3}, upon consistent rooting and placing of the buds, and a final integration.

\end{proof}

\subsection{Proof of Theorem~\ref{theo:main2} for quasi $p$-constellations}\label{sec:qconstel}
In a similar way as for quasi bipartite maps, we prove Theorem~\ref{theo:main2} in the case of quasi $p$-constellations (two boundaries have length not a multiple of $p$)
by a reduction to $p$-constellations, with some more technical details.  
We call \emph{quasi $p$-mobiles} the hypermobiles associated to quasi $p$-constellations by the bijection of Section~\ref{sec:eulerian}, 
see Figure~\ref{fig:bij_qc} for an 
example. In the following, we
 will refer to vertices  whose degree is not a multiple of $p$ as \emph{non-regular vertices} and edges whose weight is not a multiple of $p$ as \emph{non-regular edges}.

\begin{figure}
\begin{center}
\includegraphics[width=0.9\linewidth]{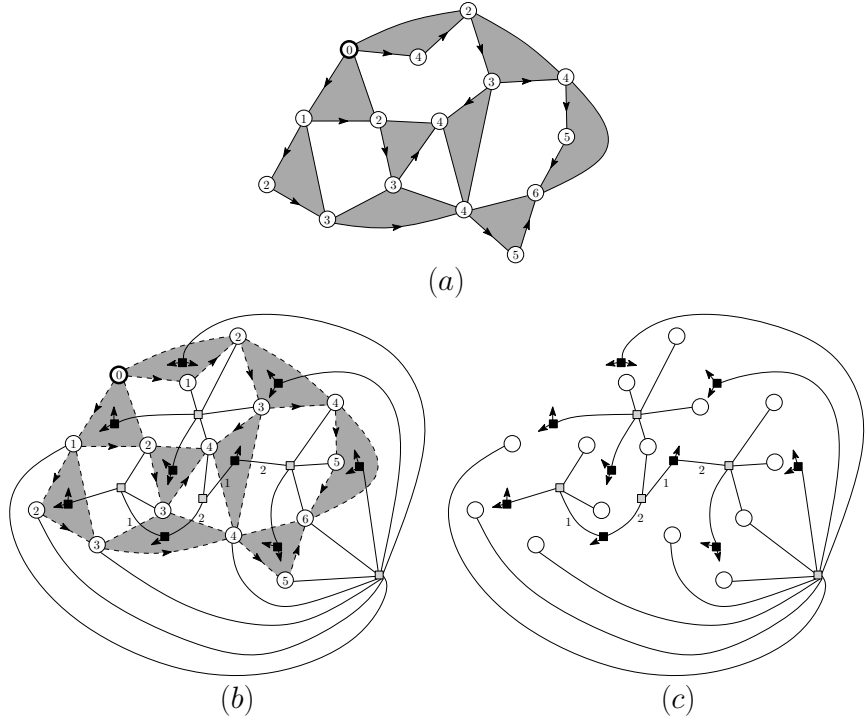}
\end{center}
\caption{(a)~A vertex-pointed quasi-$3$-constellation endowed with the  geodesic orientation. (b)~The local rule is applied to each edge of the map. (c)~The resulting quasi-$3$-mobile, where the weights on the alternating path are $(1,2,1,2)$.} 
\label{fig:bij_qc}
\end{figure}

\begin{claim}[Alternating path in a quasi-$p$-mobile]\label{claim:alt}
In a quasi-$p$-mobile, all weights of edges are at most $p$ (so regular edges have weight $p$) and 
the set of non-regular edges forms a non-empty path whose extremities are the two non-regular vertices. 
Moreover, if the degrees of the non-regular vertices $v_1,v_2$ are $pi-d=p(i-1)+p-d$ 
and $pj+d, ~i,j\geq 1, ~1\leq d\leq p-1$ (the sum of the two degrees must be a multiple of $p$), 
the weights along the path from $v_1$ to $v_2$ start with $p-d$, alternate between $p-d$ and $d$, and end with $d$. 
\end{claim}

\begin{proof}
The fact that the weights are at most $p$ just follows from the fact that dark square vertices have degree $p$. 
Let $T$ be a quasi-$p$-mobile, and let $F$ be the forest formed by the non-regular edges of $T$.  
Leaves of $F$ are necessarily non-regular, hence $F$ has only two leaves which are $v_1,v_2$, so $F$
is reduced to a path $P$ connecting $v_1$ and $v_2$. 
Starting from $v_1$, the first edge of $P$ must have weight $p-d$. 
This edge is incident to a black square vertex of degree $p$, so the following edge of $P$ must have weight $d$. The next vertex on $P$ is either $v_2$
or is a regular light square vertex, in which case the next edge along $P$ must have weight $p-d$. The alternation continues the same way until reaching $v_2$
(necessarily using an edge of weight $d$). 
\end{proof}
As for $p$-mobiles, weights on regular edges (always equal to $p$) can be omitted, and dark square vertices not on the alternating path can 
be seen as ``big buds'' (those on the alternating path are considered as ``intermediate'' dark square vertices). It is easy to check
that regular light square vertices of degree $pi$ are adjacent to $i$ big buds, and non-regular light square vertices of degree $pi+d$ (for some $1\leq d\leq p-1$)
are adjacent to $i$ big buds. 

\begin{defin}[Blossoming $p$-trees~\cite{BMS00}]
For $p\geq 2$, a \emph{planted $p$-tree} is a planted tree (non-leaf vertices are light square, leaves are round)
where the arity of internal vertices is of the form $(p-1)i$. A \emph{blossoming $p$-tree} is a structure obtained from a planted $p$-tree where:
\begin{itemize}
\item
on each edge going down to a light square vertex, a dark square vertex (called \emph{intermediate}) is inserted that additionally carries $p-2$ buds,
\item
at each light square vertex of arity $(p-1)i$ one further attaches $i-1$ new dark square vertices (called \emph{big buds}),
each such dark square vertex carrying additionally $p-1$ buds. (After these attachments, the light square vertex is considered to have degree $pi$.) 
\end{itemize}
\end{defin}

\begin{figure}
\begin{center}
\includegraphics[width=6cm]{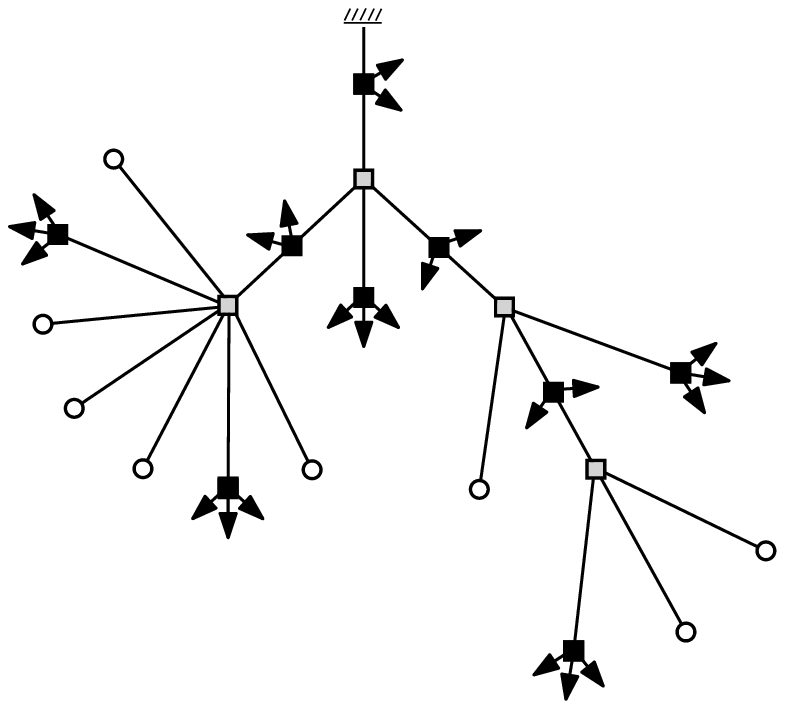}
\end{center}
\caption{A blossoming $4$-tree.}
\label{fig:blossoming_ptree}
\end{figure}

Note that in a blossoming $p$-tree, dark square vertices have degree $p$. When $p=2$, dark square vertices can be erased, 
and we obtain the description of a standard blossoming tree.
By a decomposition at the root~\cite{BMS00}, the generating function $T_p:=T_p(t;x_1,x_2,\dots)$ of rooted blossoming $p$-trees, 
where $t$ marks the number of non-root (round) leaves and $x_{i}$ marks the number of light square vertices of degree $pi$, is given by:
\begin{equation}\label{eq:Tp}
T_p=t+\sum_{i \geq 1} (p-1) \cdot x_{i} \binom{pi-1}{i-1} T_p^{(p-1)i} = t+\sum_{i \geq 1} x_{i} \binom{pi-1}{i} T_p^{(p-1)i},
\end{equation}
where the factor $(p-1)$ in the sum represents the number of ways to place the $(p-2)$ buds at the dark square vertex adjacent to the root. 

\begin{claim}\label{claim:TRp}
There is a bijection between the family $\cT_p$ of blossoming $p$-trees and the family $\cR_p$ of rooted $p$-mobiles. 
For $\gamma\in\cT_p$ and $\gamma'\in\cR_p$ the associated rooted $p$-mobile, each non-root round 
leaf of $\gamma$ corresponds to a round vertex of $\gamma'$, each light square vertex of degree $pi$ in $\gamma$ corresponds to a light square vertex of degree $pi$ in $\gamma'$. 
\end{claim}
\begin{proof}
Note that the decomposition-equation~\eqref{eq:Tp} satisfied by $T_p$ is exactly the same
as the decomposition-equation~\eqref{eq:seriesRp} satisfied by $R_p$. Hence $T_p=R_p$, and 
one can easily produce recursively a bijection between $\cT_p$ and $\cR_p$ that
 sends light square vertices of degree $pi$ to light square vertices of degree $pi$, and sends non-root round leaves to round vertices.  
\end{proof}

The bijection between $\cT_p$ and $\cR_p$ will be used in order to get rid of the alternating path 
 between the non-regular two light square vertices that appear in a quasi-$p$-mobile. 
Note that, if we denote by $\cR_p'$ the family
of rooted $p$-mobiles with a marked round vertex (which does
not contribute to the number of round vertices), and by $\cT_p'$ the family 
of blossoming $p$-trees with a marked round leaf (which does
not contribute to the number of round leaves), then $\cT_p'\simeq\cR_p'$.\\

\begin{figure}
\begin{center}
\includegraphics[width=10cm]{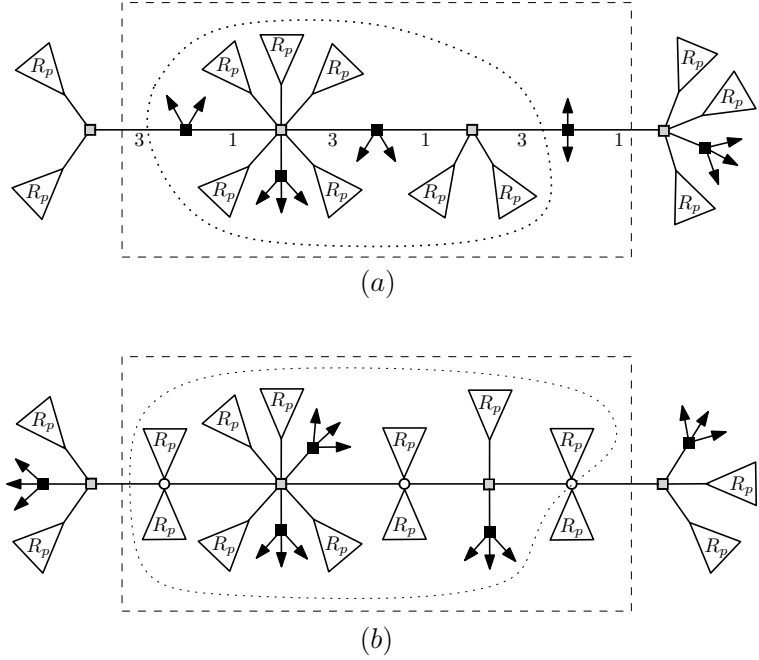}
\end{center}
\caption{Middle-parts in a quasi-$4$-mobile $(a)$, and in a $4$-mobile $(b)$.}
\label{fig:pmiddle}
\end{figure}

As in the (quasi-) bipartite case, for a hypermobile with two marked light-square vertices $v_1,v_2$,  
we can consider the operation of untying the two ends of the path $P$ connecting $v_1$ and $v_2$.
The obtained structure (taking away the connected components not containing $P$) is called the \emph{middle-part} of the hypermobile.  
Let $\cH_p$  be the family of structures that can be obtained as middle-parts of quasi-$p$-mobiles, where $v_1$ and $v_2$ are the two (ordered) non-regular vertices
(thus $P$ is the alternating path of the quasi $p$-mobile).  
And let $\cK_p$ be the family of structures  that can be obtained as middle-parts of $p$-mobiles with two marked light square vertices $v_1,v_2$. 
In the case of $\cH_p$, note that, according to Claim~\ref{claim:TRp}, the weights 
along the alternating path only depend on the degrees (modulo $p$) of the end vertices. In particular, the shape of the middle-part and the labels along the path are independent. Hence, from now on the weights can be omitted when considering middle-parts from $\cH_p$.

\begin{lem}\label{lem:pmiddle}
We have the following bijections: 
$$ \cH_p \simeq (p-1)\cdot\cT_p' \simeq (p-1)\cdot\cR_p'  \hspace{3cm} \cK_p \simeq \cR_p' \times \cR_p $$
Hence: \begin{center} $(p-1)\cdot\cK_p \simeq \cH_p \times \cR_p$. \end{center}

In these bijections, each light square vertex of degree $pi$ 
in an object on the left-hand side corresponds to a light square vertex of degree $pi$  in the corresponding object on the right-hand side.
\end{lem} 

\begin{proof}
For $\cK_p \simeq \cR_p' \times \cR_p$, the proof is similar to Lemma~\ref{lem:middle}, see Figure~\ref{fig:pmiddle}(b).  
To prove $\cH_p \simeq (p-1)\cdot\cT_p'$ (see also Figure~\ref{fig:pmiddle}(a)), 
we start similarly as in the proof of Lemma~\ref{lem:middle}, replacing each rooted $p$-mobile ``adjacent'' to the alternating path $P$ 
by the corresponding blossoming $p$-tree. Let $b$ be the (intermediate) dark square vertex adjacent to $v_2$ on $P$. 
If we erase the $p-2$ buds at $b$, then we naturally obtain a structure in $\cT_p'$ ($b$ acts as a secondary marked leaf once its incident buds are taken out).
Conversely there are $p-1$ ways to distribute the buds at $b$, which gives a factor $p-1$. 
\end{proof}  
Again, at the level of generating function expressions,
an even more precise statement (keeping track of a certain distance parameter between the two marked vertices)  
is given by Chapuy~\cite[Prop.7.5]{Ch09} (we include our quite shorter and completely bijective proof to make the paper self-contained). 

Now from Lemma~\ref{lem:pmiddle} we reduce pruned quasi-$p$-mobiles to pruned $p$-mobiles  
(pruned means: big buds at the marked light square vertices are taken out). 
Let $a_1,a_2$ be positive integers, and $1\leq d\leq p-1$. 
Define $\chB^{(p)}_{pa_1,pa_2}$ as the family of pruned $p$-mobiles with two marked black vertices $v_1, v_2$ of respective degrees $(p-1)a_1, (p-1)a_2$.
Similarly define $\chQ^{(p)}_{pa_1-d,pa_2+d}$ as the family of pruned quasi-$p$-mobiles with two marked black vertices $v_1, v_2$ of 
respective degrees $(p-1)a_1-d+1, (p-1)a_2+d$ (the two marked light square vertices are the non-regular ones).    
 
\begin{lem}\label{lem5}
For $a_1,a_2$ two positive integers, and $1\leq d\leq p-1$:
$$(p-1)\cdot\chB^{(p)}_{pa_1,pa_2} \simeq \chQ^{(p)}_{pa_1-d,pa_2+d}.$$
In addition, if $\gamma\in\chB^{(p)}_{pa_1,pa_2}$ corresponds to $\gamma'\in \chQ^{(p)}_{pa_1-d,pa_2+d}$, then each non-marked light square vertex of degree $pi$ 
 in $\gamma$ corresponds to a non-marked light square vertex of degree $pi$ in $\gamma'$.
\end{lem}

\begin{proof}
The proof is similar to Lemma~\ref{lem4}, where we additionally have to transfer a part of the degree contribution from one end of the alternating path to the other, in order to obtain a well-formed pruned $p$-mobile.
Let $\gamma \in \chQ^{(p)}_{pa_1-d,pa_2+d}$, 
and let $\tau$ be the middle-part of $\gamma$. 
We construct $\gamma'\in\chB^{(p)}_{pa_1,pa_2}$ 
as follows. Note that $v_2$ has a dark square neighbour $b$ (along
the path from $v_2$ to $v_1$) and has otherwise $(p-1)a_2+d-1$ white neighbours. Let $w_0,\dots,w_{d-1}$
be the $d$ next neighbourd after $b$ in counter-clockwise order around $v_2$, and
let $r_0,\dots,r_{d-1}$ be the mobiles (in $\cR_p$) hanging from $w_0,\dots,w_{d-1}$.  
According to Lemma~\ref{lem:middle}, 
the pair $(\tau,r_0)$ corresponds to some pair $(i,\tau')$, where $1\leq i\leq p-1$ and $\tau'\in\cK_p$. 
If we replace the middle-part $\tau$ by $\tau'$ and take out the edge $\{v_2,w_0\}$
and the mobile $r_0$, then transfer $r_1,\dots,r_{d-1}$ from $v_2$ to $v_1$, we obtain some $\gamma'\in\chB^{(p)}_{pa_1,pa_2}$.
We associate to $\gamma$ the pair $(i,\gamma')$. 
The inverse process is easy to describe, so we obtain a bijection
between $\chQ^{(p)}_{pa_1-d,pa_2+d}$ and $(p-1)\cdot\chB^{(p)}_{pa_1,pa_2}$. 
\end{proof}

Denote by $\cQ_{pa_1-d_1,pa_2+d,pa_3,\ldots,pa_r}$ the family of quasi $p$-constellations where the marked light faces are of degrees $pa_1-d,pa_2+d,pa_3,\ldots,pa_r$. 
As a corollary of Lemma~\ref{lem5} (the additionnal factors correspond to the number of ways to place
the big buds at the pruned marked vertices), 
we obtain 
$${pa_1-1 \choose a_1}{pa_2-1 \choose a_2}\cQ^{(p)}_{pa_1-d,pa_2+d}\simeq (p-1)\cdot{pa_1-d-1 \choose a_1-1}{pa_2+d-1 \choose a_2}\cB^{(p)}_{pa_1,pa_2},$$ 
and very similarly
(since the isomorphism of  Lemma~\ref{lem5} preserves light square vertex degrees):
 \begin{dmath*}
{pa_1-1 \choose a_1}{pa_2-1 \choose a_2}\cQ^{(p)}_{pa_1-d,pa_2+d,pa_3,\ldots,pa_r}\simeq (p-1)\cdot{pa_1-d-1 \choose a_1-1}{pa_2+d-1 \choose a_2}\cB^{(p)}_{pa_1,pa_2,pa_3,\ldots,pa_r},
\end{dmath*}
which yields Theorem~\ref{theo:main2} in the case where at least one of the two non-regular (degree not multiple of $p$) light faces is of degree larger than $p$. 
In the remaining we show the formula of Theorem~\ref{theo:main2} when the two non-regular light faces are of degree smaller than $p$.   

\begin{lem}\label{lem:Qp11}
Let $\cB_{p}$ be the family of $p$-mobiles with a marked light square vertex
of degree $p$, and let $\cB_p'$ be the family of objects from $\cB_p$ where
a round vertex is marked. Then, for any $d\in[1..p-1]$, 
$$
\cQ^{(p)}_{d,p-d}\simeq \cB_p'.
$$
In addition, if $\gamma\in\cB_{p}'$ corresponds to $\gamma'\in \cQ^{(p)}_{d,p-d}$, then each non-marked light square vertex of degree $pi$
in $\gamma$ corresponds to a non-marked light square vertex of degree $pi$ in $\gamma'$.
\end{lem}

\begin{proof}
A mobile in $\cQ^{(p)}_{d,p-d}$ can be decomposed as follows: two marked light squares $v_1, v_2$, their incident rooted $p$-mobiles (one for each round neighbour) 
and the middle-part. Hence we have:
\begin{eqnarray*}
\cQ^{(p)}_{d,p-d} & \simeq & \cR_p^{d-1}\times\cH_p\times\cR_p^{p-d-1} \\
 & \simeq & (p-1)\cdot \cT_p'\times\cR_p^{p-2} \\
 & \simeq & (p-1)\cdot \cR_p'\times\cR_p^{p-2}.
\end{eqnarray*}
If we now consider an object $\gamma'\in\cB_p'$, the marked light square vertex (of degree $p$) carries one big bud,
and has $p-1$ white neighbours $w_1,\ldots,w_{p-1}$. From each white neighbour $w_i$ hangs a rooted $p$-mobile $r_i$, 
and one of these rooted $p$-mobiles has a secondary marked round vertex
(the secondary marked vertex of $\gamma'$). Thus 
$$
\cB_p'\simeq (p-1)\cdot \cR_p'\times\cR_p^{p-2},
$$ 
where the factor $p-1$ is due to the choice of which of the mobiles $r_1,\ldots,r_{p-1}$ carries the secondary marked round vertex. 
\end{proof} 

 By Lemma~\ref{lem:Qp11} we have: $$p\ \!G^{(p)}_{d,p-d}=d(p-d)\ (G^{(p)}_{p})',$$ (the additional factors are due to marking a corner in each marked light face), 
and similarly:
 $$p\ \!G^{(p)}_{d,p-d,pa_3,\ldots,pa_r}=d(p-d)\ G^{(p)}_{p,pa_3,\ldots,pa_r}\ \!\!'.$$ Hence, again
 the fact that $G^{(p)}_{d,p-d,pa_3,\ldots,pa_r}$ satisfies~\eqref{eq:Fp} follows
 from the fact (already proved) that $G^{(p)}_{p,pa_3,\ldots,pa_r}$ satisfies~\eqref{eq:Fp}. This concludes the proof of Theorem~\ref{theo:main2}.

\vspace{.4cm}

\noindent\emph{Acknowledgements.} 
The authors are very grateful to Marie Albenque, Guillaume Chapuy, Dominique Poulalhon, Juanjo Ru\'e, and Gilles
Schaeffer for inspiring discussions. Both authors are supported by the  ERC grant no 208471 - ExploreMaps project. 
The second author is also supported by the ANR grant  
``Cartaplus'' 12-JS02-001-01 and the ANR grant ``EGOS'' 12-JS02-002-01.


\begin{thebibliography}{99}

\bibitem{AigZie}
M.~Aigner and G.~M. Ziegler.
\newblock \emph{Proofs from the book, 3rd edition}.
\newblock Springer, 2004.

\bibitem{BernardiFusy11}
O.~Bernardi and E.~Fusy.
\newblock Unified bijections for maps with prescribed degrees and girth.
\newblock \emph{J. Combin. Theory Ser. A}, 119(6), 1351--1387, 2012.

\bibitem{BJ05a}
M.~Bousquet-M\'elou and A.~Jehanne.
\newblock Polynomial equations with one catalytic variable, algebraic series,
  and map enumeration.
\newblock \emph{J. Combin. Theory Ser. B}, 96\penalty0 (5):\penalty0 623--672,
  2005.

\bibitem{BMS00}
M.~Bousquet-M\'elou, S.~Schaeffer.
\newblock Enumeration of planar constellations.
\newblock \emph{Advances in Applied Math},  24:\penalty0 337--368, 2000. 

\bibitem{BMS02}
M.~Bousquet-M\'elou, S.~Schaeffer.
\newblock The degree distribution in bipartite planar maps: applications to the Ising model.
\newblock arXiv:math/0211070,  2002. 

\bibitem{BDG02}
J.~Bouttier, P.~Di~Francesco, and E.~Guitter.
\newblock Census of planar maps : from the one-matrix. model solution to a combinatorial proof.
\newblock \emph{Nucl. Phys.}, B 645, 477-499, 2002.

\bibitem{BDG04}
J.~Bouttier, P.~Di~Francesco, and E.~Guitter.
\newblock Planar maps as labeled mobiles.
\newblock \emph{Electronic J. Combin.}, 11:\penalty0 R69, 2004.

\bibitem{BG13}
J.~Bouttier and E.~Guitter.
\newblock A note on irreducible maps with several boundaries. 
\newblock \emph{Electronic J. Combin.}, 21:1, 2014. 

\bibitem{Ch09}
G.~Chapuy.
\newblock Asymptotic enumeration of constellations and related families of maps
  on orientable surfaces.
\newblock \emph{Combinatorics, Probability, and Computing}, 18\penalty0
  (4):\penalty0 477--516, 2009.

\bibitem{CoFu12}
G.~Collet and \'E. Fusy.
\newblock A simple formula for the series of bipartite and quasi-bipartite maps with boundaries.
\newblock 24th International Conference on Formal Power Series 
              and Algebraic Combinatorics (FPSAC 2012),
\newblock \emph{Discrete Math. Theor. Comput. Sci. Proc., AR}, 607--618, 2012. 

\bibitem{Co75}
R.~Cori.
\newblock Un code pour les graphes planaires et ses applications.
\newblock \emph{Ast\'erisque}, 27:\penalty0 1--169, 1975.

\bibitem{Co76}
R.~Cori.
\newblock Planarit\'e et alg\'ebricit\'e.
\newblock \emph{Ast\'erisque}, 38-39:\penalty0 33--44, 1976.

\bibitem{CoVa81}
R.~Cori and B.~Vauquelin.
\newblock Planar maps are well labeled trees.
\newblock \emph{Canad. J. Math.}, 33(5):\penalty0 1023--1042, 1981.

\bibitem{Eynard11}
B.~Eynard.
\newblock \emph{Counting surfaces}.
\newblock Springer, 2011.

\bibitem{Kri07}
M.~Krikun.
\newblock Explicit enumeration of triangulations with multiple boundaries.
\newblock \emph{Electronic J. Combin.}, v14 R61, 2007.

\bibitem{Nedela07}
R.~Nedela.
\newblock Maps, hypermaps and related topics.
\newblock www.savbb.sk/\mytilde nedela/CMbook.pdf, 21--24, 2007.

\bibitem{Pitman11}
J.~Pitman.
\newblock Coalescent random forests.
\newblock \emph{J. Comb. Theory, Ser. A}, 85\penalty0 (2):\penalty0 165--193,
  1999.

\bibitem{S-these}
G.~Schaeffer.
\newblock \emph{Conjugaison d'arbres et cartes combinatoires al\'eatoires}.
\newblock PhD thesis, Universit\'e Bordeaux~I, 1998.

\bibitem{Sc97}
G.~Schaeffer.
\newblock Bijective census and random generation of {E}ulerian planar maps with
  prescribed vertex degrees.
\newblock \emph{Electron. J. Combin.}, 4\penalty0 (1):\penalty0 \string# 20, 14
  pp., 1997.

\bibitem{SiWo08}
D.~Singerman and J.~Wolfart.
\newblock Cayley graphs, Cori hypermaps, and dessins d'enfants.
\newblock \emph{Ars Mathematica Contemporanea}, Vol 1, No 2:\penalty0 144--153, 2008.

\bibitem{Tutte62}
W.~T. Tutte.
\newblock A census of slicings.
\newblock \emph{Canad. J. Math.}, 14:\penalty0 708--722, 1962.

\bibitem{Tutte63}
W.~T. Tutte.
\newblock A census of planar maps.
\newblock \emph{Canad. J. Math.}, 15:\penalty0 249--271, 1963.

\end{thebibliography}
\end{document}